\documentclass[final, times]{class2}

\usepackage[tbtags]{amsmath}

\usepackage{geometry}
\usepackage{pifont}
\usepackage{graphicx}
\usepackage{amsthm,amsmath}
\usepackage{amssymb}
\usepackage{ifthen}
\usepackage{txfonts}
\usepackage{fleqn}
\usepackage{hyperref}
\usepackage{ulem}
\usepackage{soul}
\usepackage{showkeys}

\newcommand{\R}{\mathbb R}

\newcommand{\N}{\mathbb N}

\renewcommand{\P}{\mathbb P}

\newtheorem{theo}{Theorem}[section]

\newtheorem{lem}[theo]{Lemma}

\newtheorem{korr}[theo]{Corollary}
\newtheorem{remark}[theo]{Remark}

\newcommand{\BT}{\begin{theo}}
\newcommand{\ET}{\end{theo}}
\newcommand{\BL}{\begin{lem}}
\newcommand{\EL}{\end{lem}}
\newcommand{\EQN}{\end{eqnarray}}
\newcommand{\BQN}{\begin{eqnarray}}
\newcommand{\BQNY}{\begin{eqnarray*}}
\newcommand{\EQNY}{\end{eqnarray*}}

\newcommand{\equaldis}{\stackrel{d}{=}}
\def\IF{\infty}
\newcommand{\inr}{\in \R}
\newcommand{\abs}[1]{\lvert #1 \rvert}

\usepackage{amssymb,color}
\definecolor{c20}{rgb}{0.,0.7,0.}
\definecolor{c30}{rgb}{0.,0.,1.}
\definecolor{c40}{rgb}{1,0.1,0.7}
\definecolor{c50}{rgb}{1,0,0}

\def\pE#1{\textcolor{c20}{#1}}
\def\pE#1{#1}
\def\zE#1{\textcolor{c20}{#1}}
\def\zE#1{#1}

\def\dE#1{\textcolor{c30}{#1}}

\def\dB#1{\textcolor{c50}{#1}}

\def\dE#1{#1}

\def\dB#1{#1}

\def\AE#1{\textcolor{c50}{#1}}

\def\AE#1{#1}

\def\kE#1{\textcolor{c50}{#1}}
\def\kE#1{#1}

\def\RB#1{\textcolor{c50}{#1}}
\def\RB#1{#1}

\begin{document}

\begin{frontmatter}

  \title{Extremal behavior of squared Bessel processes attracted by the Brown-Resnick process}

  \author[sutd]{Bikramjit Das\corref{cor1}}
  \ead{bikram@sutd.edu.sg}

  \author[unil,epfl]{Sebastian Engelke}
  \ead{sebastian.engelke@unil.ch}

  \author[unil]{Enkelejd Hashorva}
  \ead{enkelejd.hashorva@unil.ch}

  \cortext[cor1]{Corresponding author}
 \address[sutd]{Singapore University of  Technology and Design, 20 Dover Drive, Singapore 138682}
  \address[unil]{University of Lausanne, Faculty of Business and Economics (HEC Lausanne),
Lausanne 1015, Switzerland}
  \address[epfl]{Ecole Polytechnique F\'ed\'erale de Lausanne, EPFL-FSB-MATHAA-STAT, Station 8, 1015 Lausanne, Switzerland}

  \begin{abstract}
    The convergence of properly time-scaled and normalized maxima of independent standard
    Brownian motions to the Brown-Resnick process
    is well-known in the literature. In this paper, we study the extremal functional behavior
    of non-Gaussian processes, namely
    squared Bessel processes and scalar products of Brownian motions.
    It is shown that maxima of independent samples of those processes
    converge weakly on the space of continuous functions to the Brown-Resnick process.
  \end{abstract}

  \begin{keyword}
    Bessel process \sep Brown-Resnick process \sep extreme value
    theory \sep functional convergence
  \end{keyword}

\end{frontmatter}

\section{Introduction}\label{sec:intro}

The study of Gaussian processes, their suprema and sojourns has been of interest to researchers for quite some time; see the excellent monographs by \citet{leadbetter1983extremes}, \citet{adler:1990}, \citet{berman:1992}, \citet{lifshits:1995}, \citet{Pit96}, {\citet{AdlerTaylor}} for a detailed overview.
These studies involve investigations of the asymptotic behavior of the maximum of a Gaussian (and sometimes non-Gaussian) process over a specific set under time and space scalings. On the other hand,
in spatial extreme value theory, the main focus is on {\it{pointwise maxima}} {of}
independent processes representing regular measurements of an environmental
quantity, for instance.\\
\RB{Suppose a large number, $n$,  of} particles start at the origin and
move along the trajectories of independent Brownian motions in \RB{an}
$m$-dimensional Euclidean space. Denote by $M_n(t)$, $t\geq 0$, the
maximal squared displacement from the origin of those $n$ particles at
time $t$. It is well-known that for a fixed $t > 0$ and normalizing
sequences $a_n >0$, $b_n\in\R$, we have the weak convergence
\begin{align}\label{M_n_conv}
  \lim_{n\to\infty} \P\left(\frac{{{M_n}}(t)- b_n t}{{a_n} t} \leq x \right) \zE{=} \Lambda(x), \qquad x\in\R,
\end{align}
where $\Lambda(x)=\exp(\exp(-x))$, $x \in \R$, denotes the Gumbel
distribution; see \pE{e.g.,} \cite[p.156]{embrechts:kluppelberg:mikosch:1997}.
In this paper we are interested in the functional convergence
of the quantity in \eqref{M_n_conv} on the space of continuous
functions. In the one-dimensional case, \citet{brown:resnick:1977} showed that
the functional limit is given by a stationary, max-stable process. This Brown-Resnick
process and its generalizations in \citet{kabluchko:schlather:dehaan:2009} and \citet{kabluchko:2011}
are now well-known in extreme value theory and have
recently found importance as models for spatial extreme weather events; see
\citet{davis:klupelberg:steinkohl:2011,davison:padoan:ribatet:2012, engelkeetal:2012b}.\\
The finite-dimensional distribution of a Brown-Resnick process can be naturally
\RB{identified as} the so-called H\"usler-Reiss distributions introduced in \citet{husler:reiss:1989}
\RB{which appear }as the limit of maxima of a triangular array of Gaussian random vectors.
Those distributions arise even in more general, non-Gaussian settings,
as shown in \citet{Hashorva:2008} and \citet{hashorva:kabluchko:wubker:2010}.
In fact the latter paper provides conditions for the weak convergence of
maxima of independent, multivariate \pE{chi-square} random vectors to the
H\"usler-Reiss distribution.
\RB{Such an observation naturally points us towards the question}  whether there are some non-Gaussian processes whose maxima are
attracted by the Brown-Resnick process \RB{under appropriate linear scaling}.

This is the principal focus of our paper which is organized as follows.
In Section \ref{sec:two} we introduce necessary notation, recall the definition of Brown-Resnick processes and
provide the two main theorems. They state the functional convergence
of maxima of independent squared Bessel processes and, furthermore, it is shown that
the Brown-Resnick process also appears as the limit of maxima processes
obtained by the scalar product of two independent, $m$-dimensional Brownian motions.
The main lemma, which might be of some independent interest, and the proofs
of the theorems are relegated to Section \ref{sec:proofs}.
Section \ref{sec:conclude} concludes the paper.
Further necessary tools can be found in the Appendix.

\section{Extremal behavior of squared Bessel processes and Brownian scalar product processes } \label{sec:two}

In the sequel, for $T>0$ we denote by $C[0,T]$ and $C[0,\infty)$ the space of real-valued continuous \AE{functions} on $[0,T]$
and $[0,\infty)$, respectively, {equipped} with the topology of uniform convergence (on bounded intervals).\\
Let $\{X_i, i\in\N\}$ be the points of a Poisson point process on $\R$ with intensity measure $e^{-x}dx,\, x \in \R$, and let {$\{B_i, i\in\N\}$ be independent}  standard \AE{Brownian} motions
on $[0,\infty)$ which are also independent of $\{X_i, i\in\N \}$.
The {original Brown-Resnick process initially presented in \citep{brown:resnick:1977} is denoted by} $M_B$ and defined as
\begin{align} \label{eq:originalBR}
M_B(t)=\max_{{i\in\N}}  \Bigl( X_i + B_i{(t)} - t/2\Bigr), \qquad t\geq 0.
\end{align}
{More generally, for a centered Gaussian process $\{\eta(t),t\inr \}$ with stationary increments and \AE{variance} function $\sigma^2(t)$ the corresponding max-stable, stationary Brown-Resnick process $M_\eta$ is defined by
\begin{align}\label{eq:standardBR}
  M_\eta(t)=\max_{{i\in\N}}\Bigl(X_i + \eta_i(t) - \sigma^2(t)/2\Bigr), \qquad t\geq 0,
\end{align}
where $\eta_i,i\in\N$, are independent and identically distributed (i.i.d.) copies of $\eta$, {{see \citet{kabluchko:schlather:dehaan:2009}, \citet{kabluchko:2011}, {\citet{dombry:eyi-menko:2012}}.}\\
Originally, the standard Brown-Resnick \kE{process} {was} derived as the limit of the maximum of i.i.d.\ Gaussian processes, namely
 Brownian motions and Ornstein-Uhlenbeck processes.
 Motivated by the recent findings of {\citet{hashorva:kabluchko:wubker:2010}}, in this section we show two other classes of non-Gaussian
processes, leading to the same limit process {{$M_{B}$}}. More precisely, we investigate \pE{chi-square}, or squared Bessel processes,
and scalar-product processes related to standard Brownian motions.\\
Let therefore $\left\{B_{ij}, i \in\N, 1\le j\le m\right\}$ be independent standard Brownian motions on $[0,\IF)$ and define for $i\in\N$
the squared Bessel process of dimension $m\geq 1$ as
\begin{align}
  \xi_i(t) & =B_{i,1}^2(t)+ \ldots +B_{i,m}^2(t), \qquad t\geq 0.
\end{align}
Hence,  $\{\xi_i, i\in\N\}$ are i.i.d.\ with one-dimensional marginals given by a $\chi^{2}_m$-distribution with $m$ degrees of freedom.
Then, for constants $a_n, b_n$ defined by
\begin{align}\label{def_constA}
{a_n = 2,} \quad b_n & = 2\ln n + (m-2) \kE{\ln( \ln n)}  - {2}\ln \Gamma\left(m/2\right), \quad n\ge 1,
\end{align}
the maximum $M_{n,{\xi}}(t)=\max\{\xi_1(t),\ldots,\xi_n(t)\}$ for {any} fixed $t > 0$ satisfies
\begin{align}\label{co:chi}
\lim_{n\to\infty}\P\left(\frac{{{M_{n,\xi}}}(t)- b_n t}{{a_n} t}\le x \right)=\Lambda(x), \qquad x\in\R.
\end{align}
\pE{In} their paper, \citet{hashorva:kabluchko:wubker:2010} {prove that}
the {normalized} maxima of independent \pE{chi-square} random vectors converges \pE{to the} H\"usler-Reiss distribution \citep{husler:reiss:1989} which are the finite dimensional distributions of the Brown-Resnick processes
$M_\eta$ defined above.}
On the other hand, \citet{brown:resnick:1977} \AE{showed} that \dE{the} rescaled \kE{maxima} of an independent sequence of rescaled Brownian motions tends to the Brown-Resnick process. Thus a Brown-Resnick limit for the \kE{maxima} of  squared Bessel processes is quite intuitive.
The \dE{sequence of processes} {{$M_{n,\xi},n\ge 1$}}, is defined on $C[0,\IF)$, but {{weak convergence of $M_{n,\xi}$}} holding on $C[0,T]$ for all $T>0$  implies convergence on $C[0,\IF)$  and proving convergence on $C[0,T]$ is similar to proving it for $C[0,1]$; see \citet{brown:resnick:1977}. For the sake of simplicity, we thus show weak convergence \kE{only on} $C[0,1]$.\\
For $1 \le i\le n$, $n\in\N$, define the local, or rescaled, processes}
\begin{align}\label{eq:local}
\xi_{i,n}(t)=\frac{\xi_i\left(1+t/{b_n}\right)- b_n\left(1+ t/{b_n}\right)}{2}
,\qquad t\ge 0.
\end{align}

Our first result below shows the functional convergence of the maximum process
$\AE{\max_{1 \le i \le n} {\xi}_{i,n}}$ to the standard Brown-Resnick process $M_B$.

\BT
\label{thm1}
We have the weak convergence, as $n\to\infty$,
\begin{align*}
  \max_{i=1,\dots,n} \xi_{i,n}(t)
  \stackrel{d}{\to } M_B(t), \quad t\in[0,1]
\end{align*}
on the space of continuous functions $C[0,1]$.

\ET

\def\anY{{a^*_n}}
\def\bnY{{b^*_n}}

Since Bessel processes are the norm of multivariate Brownian motions, we shall investigate further
the extremal behavior of the scalar product of two independent Brownian motion vector processes. Let therefore
 $\{ B_{i,j}, \widetilde{B}_{i,j}, i \in\N, 1\le j \le  m\}$ be independent standard Brownian motions on $[0,\IF)$ and define for $i\in\N$
\begin{align}
\gamma_i(t) & =B_{i,1}(t)\widetilde{B}_{i,1}(t)+ \ldots +B_{i,m}(t)\widetilde{B}_{i,m}(t), \quad t\in [0,\IF).
\end{align}
By Lemma \ref{lemProducts} in the Appendix it follows that for constants $\anY,\bnY$ defined by
\begin{align}\label{def_constB}
{\anY = 1,} \quad \bnY & = \ln n + (m/2-1) \ln( \ln n) - (m/2-1) \ln 2 - \ln \Gamma\left( m /2\right), \quad n\ge 1
\end{align}
the maximum process $M_{n,{\gamma}}(t)=\max\{\gamma_1(t),\ldots,\gamma_n(t)\}$ for a fixed $t > 0$ satisfies
\begin{align}\label{co:prod}
\lim_{n\to\infty}\P\left(\frac{M_{n,{\gamma}}(t)- \bnY t}{\anY t}\le x \right)=\Lambda(x), \qquad x\in\R.
\end{align}
Note in passing that $\anY, \bnY$ are however different than in the case of squared Bessel processes.
Similarly as above we define for $1 \le i\le n$ the local processes
\begin{align}\label{eq:local2}
\gamma_{i,n}(t)= \gamma_i\left(1+t/(2\bnY)\right)- \bnY\left(1+ t/(2\bnY)\right)
,\qquad t\ge 0.
\end{align}
We have the following result for the convergence of $\AE{\max_{1 \le i \le n} {\gamma}_{i,n}}$, as $n\to\infty$.

\BT
\label{thm2}
For $n\to\infty$, we have the weak convergence
\begin{align*}
  \max_{i=1,\dots,n} \gamma_{i,n}(t)
  \stackrel{d}{\to } M_{B}(t), \quad t\in[0,1]
\end{align*}
on the space of continuous functions $C[0,1]$.
\ET

\section{Proofs}\label{sec:proofs}

Let us remark, that the space $C[0,1]$ of continuous functions is not
locally compact. This fact prevents us to apply the standard
theory for Poisson point processes in extreme value theory.
In particular,  \cite[Theorem 5.3]{resnickbook:2007} is not applicable
for Poisson point processes on the space $C[0,1]$. We thus rely
on a similar technique as in the proof of Theorem 17 in \citet{kabluchko:schlather:dehaan:2009}
in order to show negligibility of lower order statistics.

We first prove the following main lemma, which \kE{is of some} independent
interest as a tool for showing weak convergence to the Brown-Resnick
process. For instance, it implies the weak convergence results in
\citet{brown:resnick:1977}.

\BL\label{main_lemma}
For $n\in\N$, $1\leq i \leq n$, let the following triangular arrays
be given, where the elements within the rows of each array are i.i.d.:
\begin{enumerate}
\item\label{assump1}
  Identically distributed random variables $Y_{i,n}$ \kE{satisfying}
  \BQN\label{Y11}
    \P( Y_{1,1} \kE{>u} ) = (1+o(1))K u^\beta e^{-cu}du, \quad u\to \IF,
  \EQN
  with constants $K, c > 0$, $\beta\in\R$.
   By Theorem 3.3.26 in \citet{embrechts:kluppelberg:mikosch:1997}, this implies that
  \BQN\label{Gumbel_conv}
    \lim_{n\to\infty} n\P( X_{1,n} \kE{>s}) =e^{-s},  \quad \forall s \inr,
  \EQN
  where $X_{i,n} = a_n^{-1}( Y_{i,n} - b_n )$ and
  \BQN\label{def_constC}
    {a_n = c^{-1},} \quad b_n = c^{-1}\left(\ln n +
    \beta \ln(c^{-1} \ln n) + \ln K \right), \quad n\ge 1.
  \EQN
  \kE{Assume further that for all large $r$ and any $p>0$}
   \BQN\label{conditionKK}
   \limsup_{n\to \IF} n\int_{-b_n/(2a_n)}^{-r}
e^{- x^2/p}\P(X_{i,n} \in dx)< \IF.
  \EQN

\item\label{assump2}
  Stochastic processes $\{ R_{i,n}(t): t \in [0,1]\}$, such that
  the vector $(X_{i,n},R_{i,n}(\cdot))$ has the same distribution
  as $(X_{i,n},\phi_{i,n}W_{i,n}(\cdot))$, where $W_{i,n}\sim \{W(t) : t\in[0,1]\}$
  are standard Brownian motions independent of the $X_{i,n}$, and $\phi_{i,n}$ are positive random variables, independent of $W_{i,n}$ \RB{such} that for some $q>0$
  \BQN\label{negl_cond_phi_n}
    \lim_{n\to\infty} n\P( \phi_{1,n} > q) = 0,
  \EQN
  and
  \BQN\label{negl_cond_phi}
    \lim_{n\to\infty} \P( | 1 - \phi_{1,n} | > \epsilon \,|\, X_{1,n} \in K ) = 0, \quad \forall \epsilon > 0
  \EQN
  for any compact set $K\subset \R$.

\item\label{assump3}
  Stochastic processes $\{ \delta_{i,n}(t): t \in [0,1]\}$, independent of $X_{i,n}$, such that
  \BQN
    \label{negl_cond_delta}&\lim_{n\to\infty} \P( \| \delta_{1,n}\| > \epsilon ) &= 0, \quad \forall \epsilon > 0,\\
    \label{negl_cond_delta_n}&\lim_{n\to\infty} n\P( \| \delta_{1,n}\| > C ) &= 0, \text{ for some } C>0.
  \EQN
\end{enumerate}
Then, we have the weak convergence
\begin{align}\label{BR_conv}
  \eta_{n}(t) := \max_{i = 1,\dots, n} \zeta_{i,n}(t) := \max_{i = 1,\dots, n} \Bigl(X_{i,n} + R_{i,n}(t) -t/2 + \delta_{i,n}(t)\Bigr)
  \stackrel{d}{\to } M_B(t), \quad t\in[0,1],
\end{align}
on the function space $C[0,1]$, where $\{M_B(t) : t\in[0,1]\}$ is the original
Brown-Resnick process \RB{given by \eqref{eq:originalBR}}.
\EL

\begin{remark} \label{rem32}
{If \eqref{Y11} holds, then condition \eqref{conditionKK} is satisfied if $Y_{1,1}$ possesses a density $h$ such that
for some $c>0$
$$\P(Y_{1,1}> u) = (1+o(1)) h(u)/c, \quad u\to \IF.$$}
\end{remark}
We will prove \pE{frist} the following useful result.

\BL\label{lemma_lower}
With the notation and under the assumptions of Lemma \ref{main_lemma},
for any $\epsilon > 0$ we can find constants $R,N > 0$ such that
for any $r\pE{>} R$ and $n\pE{>} N$, we have
\begin{align}\label{def_A_n}
  \P( A_n ) := \P\left( \exists t\in[0,1]: \eta_n(t) \neq \max_{i\in\{1,\dots,n\}, |X_{i,n}| < r} \zeta_{i,n}(t) \right) \leq \epsilon.
\end{align}
\EL

\begin{proof}
  We apply a similar technique as in the proof of Theorem 17 in \citet{kabluchko:schlather:dehaan:2009}.
  First note that
  \begin{align*}
    A_n \subset  C_n \cup D_{n}
    \cup \left(A_n \setminus \left[ C_n \cap D_{n} \right] \right),
  \end{align*}
  where for some $r_1>0$
  \begin{align*}
    C_n = \left\{ \inf_{t\in[0,1]} \eta_n(t) < -r_1 \right\},\qquad D_{n} = \bigcup_{i=1}^n \left\{ X_{i,n} > r \right\}.
  \end{align*}
  Clearly by \eqref{Gumbel_conv}, for $N$ and $R$ large enough it holds
  that $\P(D_n) = n\P(X_{1,n} > r) < \epsilon$, for any $n>N$, $r>R$.
  Moreover, note that $C_{n} \subset \bigcap_{i=1}^n F_{i,n}^c$, where
  \begin{align*}
    F_{i,n} = \left\{ X_{i,n} \in [-r,r], \inf_{t\in[0,1]} R_{i,n}(t) -t/2 + \delta_{i,n}(t)
    \geq r - r_1 \right\}.
  \end{align*}
  In view of assumption \eqref{negl_cond_phi}, for
  \begin{align}\label{Delta_def}
    \Delta_{i,n}(t) := W_{i,n}(t)(\phi_{i,n} - 1), \qquad t\in[0,1]
  \end{align}
  we obtain for any $\delta >0$
  \begin{align}\label{Delta_est}
    \P(\|\Delta_{i,n}\| > \delta | X_{i,n} \in [-r,r]) &\leq \P(\|W_{i,n}\| > \delta/\tau) + \P(|(\phi_{i,n} - 1)| > \tau | X_{i,n} \in [-r,r])\\
    \notag&\leq \overline{\Phi}(\delta/\tau) + \epsilon/2\\
    \notag&\leq \epsilon
  \end{align}
  for sufficiently small $\tau>0$ with  $\overline{\Phi}$ the \pE{tail} of \pE{an} $N(0,1)$ random variable.
  Further, using assumption \ref{assump2} of Lemma \ref{main_lemma}, \eqref{negl_cond_delta} and \eqref{Delta_est}, we obtain for any $\delta > 0$
  \begin{align*}
    \P&\left(\inf_{t\in[0,1]} R_{i,n}(t) -t/2 + \delta_{i,n}(t)
    < r - r_1 \Big| X_{i,n} \in [-r,r]\right)\\
    &\qquad \leq \P(\|\delta_{i,n}\| > \delta) + \P\left( \inf_{t\in[0,1]} W_{i,n}(t) -t/2 + \Delta_{i,n}(t) <  r - r_1 + \delta \Big| X_{i,n} \in [-r,r]\right)\\
    &\qquad\leq \P(\|\delta_{i,n}\| > \delta) + \P(\|\Delta_{i,n}\| > \delta | X_{i,n} \in [-r,r]) +
\P\left( \inf_{t\in[0,1]} W_{i,n}(t) <  r - r_1 + 2\delta +1/2\right)\\
&\qquad \leq \frac12,
  \end{align*}
  for $n$ and $r_1$ sufficiently large. Thus,  {by \eqref{Gumbel_conv}}
  \begin{align*}
    \P( F_{i,n}) \geq \frac12 \, \P( X_{i,n} \in [-r,r]) \geq \frac r n + o(1/n),\quad n\to\infty
  \end{align*}
  for $r$ large enough {and uniformly in $i\in \N$}, and consequently
  \begin{align*}
    \P(C_n) \;& \RB{\le} \left(1 - \P( F_{1,n})\right)^n \leq \left(1 - \frac r n + o(1/n)\right)^n \leq 2e^{-r} < \epsilon
  \end{align*}
  for $r$ and $n$ large. It remains to show \pE{that} $\P(A_n \setminus \left( C_n \cap D_{n} \right))$
  becomes small. To this end, define events
  \begin{align*}
    E_{i,n} = \Bigl\{ X_{i,n} < -r, \sup_{t\in[0,1]} \zeta_{i,n}(t) > -r_1 \Bigr\}
  \end{align*}
  and note that $A_n \setminus \left( C_n \cap D_{n} \right)$ is a subset
  of the union $\bigcup_{i=1}^n E_{i,n}$.
  Let $C>0$ be the constant in \eqref{negl_cond_delta_n} and recall the stochastic representation of $R_{i,n}$ from assumption \ref{assump2}. Then
  \begin{align}\label{est_Ein}
    \P(E_{i,n}) \leq & \P(\|\delta_{i,n}\| > C) + \P\left( X_{i,n} < -r, \sup_{t\in[0,1]}\Bigl( X_{i,n} + \phi_{i,n}W_{i,n}(t) - t/2\Bigr) > -r_1 - C \right).
  \end{align}
  For $n$ large enough, \eqref{negl_cond_delta_n} implies that the first summand is bounded by $\epsilon / n$.
  A coupling argument yields that the second summand can be bounded from above by
  \begin{align*}
    \P(\phi_{i,n} > q ) + \P\left( X_{i,n} < -r, \sup_{t\in[0,1]} \Bigl(X_{i,n} + qW_{i,n}(t)\Bigr) > -r_1 - C \right),
  \end{align*}
  where again the first summand is bounded by $\epsilon / n$ by \eqref{negl_cond_phi_n}. Clearly, we can estimate
  \begin{align*}
    \P\left(\sup_{t\in[0,1]} W_{i,n}(t)> u \right) \leq 2\overline{\Phi}(u) \leq e^{-u^2/2}
  \end{align*}
  for large $u>0$.
  Choosing $r > 0$ large enough, such that $(-r_1 -C -x)/q > -x /(2q)$ for all $x < - r$ thus gives
  \BQNY
    \lefteqn{\P \left( X_{i,n} < -r, \sup_{t\in[0,1]} \Bigl(X_{i,n} + qW_{i,n}(t) \Bigr)> -r_1 - C \right)}\\
  &\leq &
  \P \left(\sup_{t\in[0,1]} W_{i,n}(t) > \frac{-r_1 - C + b_n/(2a_n)}{q} \right) \\
  && +
\int_{-b_n/(2a_n)}^{-r}
    \P\left(\sup_{t\in[0,1]} W_{i,n}(t) > \frac{-r_1 -x -C}{q} \,\Big|\, X_{i,n} = x\right) \P(X_{i,n} \in dx)\\
    &\leq &e^{-(b_n/(4a_nq))^2/2} +
\int_{-b_n/(2a_n)}^{-r}
    e^{- x^2/(8q)} \P(X_{i,n} \in dx)\\
    &\leq & e^{-(b_n/(4a_nq))^2/2} +
    \frac{K'}{n} 
\\    & \leq & \epsilon / n,
  \EQNY
  with some constant $K'>0$ and $a_n$ and $b_n$ defined in \eqref{def_constC}.
  The second inequality above is a consequence of \eqref{conditionKK}.
  Collecting all parts together yields
  \begin{align*}
    \P\left(\bigcup_{i=1}^n E_{i,n}\right) \leq nP\left(E_{i,n}\right) \leq \epsilon
  \end{align*}
  and thus $\P(A_n) \leq 3\epsilon$ for all $n > N$ and $r > R$ with $N,R$ large enough. \end{proof}

\begin{korr}\label{cor1}
  With the notation and under the assumptions of Lemma \ref{main_lemma},
  for any $\epsilon > 0$ we can find an $N\in\N$, such that for all $n\pE{>} N$ we have
  \begin{align*}
    \P\left( \sup_{t\in[0,1]} \left| \eta_n(t) - \tilde\eta_n(t) \right| > \epsilon \right) \leq \epsilon,
  \end{align*}
  where
  \begin{align}\label{tilde_eta}
    \tilde\eta_n(t) = \max_{i = 1,\dots,n } \Bigl(X_{i,n} + R_{i,n}(t) -t/2\Bigr), \quad t\in[0,1].
  \end{align}
\end{korr}

\begin{proof} \pE{For any $\epsilon>0$ we have}
  \BQNY
    \lefteqn{\P\left( \sup_{t\in[0,1]} \left| \eta_n(t) - \tilde\eta_n(t) \right| > \epsilon \right)}\\
    &\leq& \P\left( \exists t\in[0,1]: \eta_n(t) \neq \max_{i\in\{1,\dots,n\}, |X_{i,n}| < r} \Bigl(X_{i,n} + R_{i,n}(t) -t/2 + \delta_{i,n}(t)\Bigr) \right)\\
    &&+ \P\left( \exists t\in[0,1]: \tilde\eta_n(t)  \neq \max_{i\in\{1,\dots,n\}, |X_{i,n}| < r}
    \Bigl(X_{i,n} + R_{i,n}(t) - t/2 \Bigr)\right)\\
    &&+ \P\left( \exists i\in\{1,\dots, n\}: |X_{i,n}| < r, \|\delta_{i,n}\| > \epsilon \right)\\
    &\leq& \epsilon/3 + \epsilon/3 + n\P\left(|X_{i,n}| < r)\P(\|\delta_{i,n}\| > \epsilon\right)\\
    & \leq& \epsilon,
  \EQNY
  where for the first and second summand $r$ and $N$ can be chosen according to Lemma
  \ref{lemma_lower}. The last inequality then follows from assumptions \eqref{Gumbel_conv}
  and \eqref{negl_cond_delta}.
\end{proof}

\begin{proof}[Proof of Lemma \ref{main_lemma}]

  The proof will consist of two steps. First, we establish convergence of the
  finite dimensional margins in \eqref{BR_conv}, and, second, we show
  that the sequence of probability measures $\{ \eta_n \}_{n\in\N}$ on $C[0,1]$
  is tight. In fact, by Corollary \ref{cor1}, $\{ \eta_n \}_{n\in\N}$ converges
  weakly on $C[0,1]$ if and only if the sequence of probability measures
  $\{ \tilde\eta_n \}_{n\in\N}$ in \eqref{tilde_eta} converges weakly on $C[0,1]$,
  and, in this case, the limits are equal. In the sequel we will therefore
  consider $\{ \tilde\eta_n \}_{n\in\N}$ instead of $\{ \eta_n \}_{n\in\N}$.\\
  For the first part, let ${\bf{t}} = (t_1,\dots,t_m) \in [0,1]^m$ and
  $(y_1,\dots,y_m)\in \R^m$ be fixed. It follows
  from \kE{Lemma 4.1.3 in \citet{HRF}} that it suffices to proof the convergence
  \begin{align}\label{lim}
    \lim_{n\to\infty}  n \P( \forall j: & X_{1,n} + R_{1,n}(t_j) - t_j/2 > y_j ) = \int_\R e^{-y} \P\left(\forall j: W(t_j) - t_j/2 > y_j - y \right) \, {d y},
  \end{align}
  where $\{W(t): t\in[0,1]\}$ is a standard Brownian motion.
  To this end, we recall the definition of $\Delta_{1,n}$ in \eqref{Delta_def} and, for clarity, denote
  by $\{\overline W_{1,n}(t) = W_{1,n}(t) - t/2: t\in[0,1]\}$ the drifted process. For arbitrary $\delta, r > 0$ we obtain the estimate
  \BQN
    \notag\P( \forall j:  X_{1,n} + \overline W_{1,n}(t_j) + \Delta_{1,n}(t_j) > y_j)
    &\leq &\P(\forall j:  X_{1,n} + \overline W_{1,n}(t_j) > y_j - \delta, |X_{1,n}| < r)\\
    \notag
    && + \P(\forall j:  X_{1,n} + \overline W_{1,n}(t_j) + \Delta_{1,n}(t_j) > y_j, |X_{1,n}| > r)\\
    \label{summands}
    &&+ \P( \|\Delta_{1,n}\| > \delta, |X_{1,n}| < r).
  \EQN
   Furthermore, as $n\to\infty$, the first summand fulfills
  \begin{align*}
    n\P(\forall j:  X_{1,n} + \overline W_{1,n}(t_j) > y_j - \delta, |X_{1,n}| < r)
    &= \int_{-r}^r  \P(\forall j: \overline W_{1,n}(t_j) > y_j - y - \delta) n\P(X_{1,n} \in dy)\\
    &\to  \int_{-r}^r e^{-y} \P(\forall j: \overline W_{1,1}(t_j) > y_j - y - \delta)\, dy,
  \end{align*}
  since by \eqref{Gumbel_conv}, $n\P(X_{1,n} \in dy)$ converges weakly to $e^{-y}dy$, as $n\to\infty$. Now, in view of the calculations following \eqref{est_Ein} for the second summand in \eqref{summands},
  and \eqref{Delta_est} and \eqref{Gumbel_conv} for the third summand in \eqref{summands},
  we have
  \begin{align}
    \notag\limsup_{n\to\infty} \, nP( \forall j:  X_{1,n} &+ \overline W_{1,n}(t_j) + \Delta_{1,n}(t_j) > y_j) \\
    \notag&\leq \lim_{r\to\infty} \int_{-r}^r e^{-y} \P(\forall j: \overline W_{1,1}(t_j) > y_j - y - \delta)\, dy\\
    \notag&\qquad + \lim_{r\to\infty}\limsup_{n\to\infty} \, n\P(\forall j:  X_{1,n} + \overline W_{1,n}(t_j) + \Delta_{1,n}(t_j) > y_j, |X_{1,n}| > r)\\
    \notag&\qquad + \lim_{r\to\infty}\limsup_{n\to\infty}\, n\P( \|\Delta_{1,n}\| > \delta, |X_{1,n}| < r)\\
    \label{limsup}& = \int_\R e^{-y} \P(\forall j: \overline W_{1,1}(t_j)  > y_j - y - \delta)\, dy.
  \end{align}
  Similarly, we can show that
  \begin{align}\label{liminf}
    \liminf_{n\to\infty} \, nP( \forall j:  X_{1,n} &+ \overline W_{1,n}(t_j) + \Delta_{1,n}(t_j) > y_j) \geq \int_\R e^{-y} \P(\forall j: \overline W_{1,1}(t_j) > y_j - y + \delta)\, dy.
  \end{align}
  Since $\delta>0$ was arbitrary, \eqref{lim} follows from \eqref{limsup} and \eqref{liminf} as $\delta\searrow 0$, and thus the convergence of finite
  dimensional margins.

  In order to show the tightness of the sequence $\{ \tilde\eta_n \}_{n\in\N}$
  we note that the sequence $\{ \tilde\eta_n(0) \}_{n\in\N}$ is tight
  since it equals $\{ \max_{i=1,\dots,n} X_{i,n} \}_{n\in\N}$ which converges
  to the Gumbel distribution by \eqref{Gumbel_conv}.
  For a function $g \in C[0,1]$ and any $\kappa>0$, we define
  the modulus of continuity $\omega_\kappa(g)$
  \begin{align*}
    \omega_\kappa(g) = \sup_{s,t\in[0,1], |s-t| \leq \kappa} |g(s) - g(t)|.
  \end{align*}
  By Theorem 7.3 in \citet{billingsley:1999} it suffices to find for any
  $\epsilon, \alpha > 0$ a $\kappa > 0$ and $N\in\N$ such that
  \begin{align*}
    \P( \omega_\kappa(\tilde \eta_n) > \alpha ) < \epsilon, \quad n > N.
  \end{align*}
  By choosing $\kappa>0$ small enough, we get for any $r>0$
  \begin{align}
    \notag\P( \omega_\kappa(X_{1,n} + \overline W_{1,n} + \Delta_{i,n}) > \alpha \,|\, X_{i,n}\in[-r,r])  &\leq
    \P( \omega_\kappa(\overline W_{1,n}) > \alpha/2) +  \P(\|\Delta_{1,n}\| > \alpha/2 \,|\, X_{i,n}\in[-r,r]) \\
    \label{modulus}&\leq \epsilon/2
  \end{align}
  for all $n \pE{>} N$ with $N$ large enough, because of
  the fact that $\overline W_{1,n}$ is independent of $X_{1,n}$
  and its distribution does not depend on $n$,
  and condition \eqref{Delta_est}.
  We proceed by noting that for any $n$, we have
  \begin{align}\label{set_inclusion}
    \{\omega_\kappa(\tilde\eta_n) > \alpha \}
    \subset \left( \{\omega_\kappa(\tilde\eta_n) > \alpha \} \cap \tilde A_n^C \right)
    \cup \tilde A_n \subset \left(\bigcup_{i=1}^n G_{i,n},\right)\cup \tilde A_n,
  \end{align}
  where
  \begin{align*}
    \tilde A_n = \left\{\exists t\in[0,1]: \tilde\eta_n(t)  \neq \max_{i\in\{1,\dots,n\}, |X_{i,n}| < r}
    \Bigl(X_{i,n} + R_{i,n}(t) - t/2 \Bigr) \right\}
  \end{align*}
  and
  \begin{align*}
    G_{i,n} &= \left\{X_{i,n}\in [-r,r], \omega_\kappa(X_{i,n} +\overline W_{1,n} + \Delta_{i,n}) > \alpha \right\}.
  \end{align*}
  \pE{Conditioning} we obtain for any $\epsilon'>0$
  \begin{align*}
    \P(G_{i,n}) = \P( \omega_\kappa(X_{i,n} + \overline W_{1,n} + \Delta_{i,n}) > \alpha \,|\, X_{i,n}\in[-r,r]) \P( X_{i,n} \in [-r,r])
    \leq \epsilon' \, \P( X_{i,n} \in [-r,r])
  \end{align*}
  by \eqref{modulus} and $\kappa>0$ small enough, for any $n> N$.
  Thus, since by \eqref{Gumbel_conv}, $\P( X_{i,n} \in [-r,r])$ is
  of order $1/n$, we {have} for any $n> N$ with $N$ large enough
  \begin{align}\label{G_i_n}
    \P\left(\bigcup_{i=1}^n G_{i,n}\right) \leq n \P(G_{i,n}) < \epsilon/2.
  \end{align}
  Consequently, \eqref{set_inclusion} together with \eqref{def_A_n} and \eqref{G_i_n} implies $\P(\omega_\kappa(\tilde\eta_n) > \alpha) < \epsilon$, for $n> N$, and hence the tightness of $\{ \tilde\eta_n \}_{n\in\N}$.
\end{proof}

\begin{proof}[Proof of Theorem \ref{thm1}]   For $i\in\N$ and $1\leq j \leq m$, write
  \begin{align*}
    B_{i,j}\left(1+t/{b_n}\right) & {\equaldis}  B_{i,j}(1)+ \frac{1}{\sqrt{b_n}} B_{i,j}^*(t), \quad t\ge 0,
  \end{align*}
where $ \{B_{i,j}^*(t), i\in\N, 1\le j \le m\}$ are {independent}
standard Brownian motions {being further} independent of
$ \{B_{i,j}(1), i\in\N, 1\le j \le m\}$. We thus have
\begin{align}
  \notag\xi_{i,n}(t) &\stackrel{d}{=}   \frac 12 \left(\sum_{j=1}^m (B_{i,j}(1))^2 +
  \frac{{2}}{\sqrt{b_n}}\sum_{j=1}^m B_{i,j}(1)B_{i,j}^*(t) +
  \frac{1}{b_n}\sum_{j=1}^m (B_{i,j}^*(t))^2 - b_n\left(1+t/b_n\right)\right) \\
  \notag& =  \frac{\sum_{j=1}^m (B_{i,j}(1))^2- b_n}{2}+ \left(\frac{1}{\sqrt{b_n}}\sum_{j=1}^m B_{i,j}(1)B_{i,j}^*(t) -\frac t2\right) +
  \frac{1}{2b_n}\sum_{j=1}^m (B_{i,j}^*(t))^2\\
  \label{split}&=:  X_{i,n} + R_{i,n}(t) - t/2 + \delta_{i,n}(t), \quad t\in[0,1].
\end{align}
We check the assumptions of Lemma \ref{main_lemma}.
By Lemma \ref{lemProducts} in the Appendix, $Y_{i,n} := 2X_{i,n} + b_n$ \kE{satisfies} for $u\to\infty$
\begin{align*}
  \P\left(Y_{1,1}> \kE{u} \right) = (1+o(1)) \frac{1}{2^{ m/2} \Gamma(m/2)} \kE{u}^{m/2-1} \exp(-\kE{u}/2)= 
  \pE{(1+o(1))2\P\left(Y_{1,1} \in du \right)} 
\end{align*}
and hence assumption \ref{assump1} of Lemma \ref{main_lemma} holds (recall Remark \ref{rem32}).

A simple calculation with characteristic functions yields for $X_{i,n}$ and $R_{i,n}$ in \eqref{split}
the joint stochastic representation
\begin{align*}
  \left( X_{i,n}, R_{i,n}\right)
  \stackrel{d}{=} \left(X_{i,n}, \phi_{i,n} W_{i,n}(\cdot) \right),\qquad \phi_{i,n}:= \sqrt{2X_{i,n}/b_n + 1},
\end{align*}
where $\{W_{i,n}(t): t\in[0,1]\}$ are i.i.d.\ standard Brownian motions, independent of the $X_{i,n}$.
Clearly, it holds for any $q > 1$ that
\begin{align*}
  \lim_{n\to\infty} n\P( \phi_{1,n} > q ) = \lim_{n\to\infty} n\P( X_{1,n} > b_n(q^2 - 1)/2 ) = 0
\end{align*}
since $X_{i,n}$ is in the max-domain of attraction of the Gumbel distribution and $\lim_{n\to\infty} b_n(q^2 - 1)/2 =\IF$.
Furthermore, for arbitrary $\epsilon,r > 0$ we trivially have
\begin{align*}
  \lim_{n\to\infty} \P( |1 - \phi_{1,n}| > \epsilon | X_{1,n} \in [-r,r]) = 0.
\end{align*}
Thus, assumption \ref{assump2} of Lemma \ref{main_lemma} is fulfilled.

We note that $\delta_{i,n}$ in \eqref{split} is independent of $X_{i,n}$ and for any $\epsilon > 0$
\begin{align*}
  \P( \| \delta_{1,n}\| > \epsilon ) =
  \P\left(\sup_{t\in[0,1]} \Biggl(\sum_{j=1}^m (B_{i,j}^*(t))^2\Biggr) > 2b_n\epsilon  \right) \to 0,\quad n\to\infty.
\end{align*}
Moreover, for a $C>1$, \kE{in view of the Piterbarg inequality given in Proposition 3.2 in \citet{TanHash13} (see also Theorem 8.1 in \citet{Pit96}, or in \citet{Pit2001}), we have for some positive constant $\lambda$}
\begin{align*}
  n\P( \| \delta_{1,n}\| > C )& =
  n\P\left(\sup_{t\in[0,1]} \Biggl(\sum_{j=1}^m (B_{i,j}^*(t))^2\Biggr) > 2b_nC  \right) \\
  &\le n b_n^\lambda e^{-b_n C}\\
  &\to   0,\quad n\to\infty
\end{align*}
and thus assumption \ref{assump3} of Lemma \ref{main_lemma} holds, and the assertion of the theorem follows.
\end{proof}

\begin{proof}[Proof of Theorem \ref{thm2}]
  For $i\in\N$ and $1\leq j \leq m$, write
  \begin{align*}
    B_{i,j}\left(1+t/(2\bnY)\right) {\equaldis}  B_{i,j}(1)+ \frac{1}{\sqrt{2\bnY}} B_{i,j}^*(t),
    \quad \widetilde{B}_{i,j}\left(1+t/(2\bnY)\right) {\equaldis}  \widetilde{B}_{i,j}(1)+ \frac{1}{\sqrt{2\bnY}} \widetilde{B}_{i,j}^*(t) \quad t\ge 0,
  \end{align*}
where $ \{B_{i,j}^*(t),\widetilde{B}_{i,j}^*, i\in\N, 1\le j \le m\}$ are {independent}
standard Brownian motions {being further} independent of
$ \{B_{i,j}(1),\widetilde{B}_{i,j}, i\in\N, 1\le j \le m\}$. We thus have
\begin{align}
  \notag\gamma_{i,n}(t) &\stackrel{d}{=}  \sum_{j=1}^m B_{i,j}(1)\widetilde{B}_{i,j}(1) +
  \frac{{1}}{\sqrt{2\bnY}}\sum_{j=1}^m B_{i,j}(1)\widetilde B_{i,j}^*(t) \\
  \notag& \qquad +\frac{{1}}{\sqrt{2\bnY}}\sum_{j=1}^m \widetilde B_{i,j}(1) B_{i,j}^*(t) +
  \frac{1}{2\bnY}\sum_{j=1}^m B_{i,j}^*(t)\widetilde B_{i,j}^*(t) - \bnY\left(1+t/(2\bnY)\right) \\
  \notag&=  \left(\sum_{j=1}^m B_{i,j}(1)\widetilde{B}_{i,j}(1) - \bnY\right) + \left(\frac{{1}}{\sqrt{2\bnY}}\sum_{j=1}^m B_{i,j}(1)\widetilde B_{i,j}^*(t) + \frac{{1}}{\sqrt{2\bnY}}\sum_{j=1}^m \widetilde B_{i,j}(1) B_{i,j}^*(t) - \frac t2  \right)\\
  \notag&\qquad + \frac{1}{2\bnY}\sum_{j=1}^m B_{i,j}^*(t)\widetilde B_{i,j}^*(t) \\
  \label{split2}&=:  X_{i,n} + R_{i,n}(t) - t/2 + \delta_{i,n}(t), \quad t\in[0,1].
\end{align}
As above, we only have to check the assumptions of Lemma \ref{main_lemma}.
By Lemma \ref{lemProducts} in the Appendix, $Y_{i,n} := X_{i,n} + \bnY$ satisfies for $u\to\infty$
\begin{align*}
  \P\left(Y_{1,1}\kE{>u} \right) = (1+o(1)) \frac{1}{2^{ m/2} \Gamma(m/2)} u^{m/2-1} \exp(-\kE{u})
=    \pE{(1+o(1))\P\left(Y_{1,1} \in du \right)} 
\end{align*}
and hence assumption \ref{assump1} of Lemma \ref{main_lemma} holds {(recall again Remark \ref{rem32})}.

A simple calculation with characteristic functions yields for $X_{i,n}$ and $R_{i,n}$ in \eqref{split2}
the joint stochastic representation
\begin{align*}
  \left( X_{i,n}, R_{i,n}\right)
  \stackrel{d}{=} \left(X_{i,n}, \phi_{i,n} W_{i,n}(\cdot) \right),\qquad \phi_{i,n}:= \sqrt{\frac{\Psi_{i,n}}{2b_n}},
    \qquad \Psi_{i,n} := \sum_{j=1}^m \left(B^2_{i,j}(1) + \widetilde{B}^2_{i,j}(1)\right),
\end{align*}
where $\{W_{i,n}(t): t\in[0,1]\}$ are i.i.d.\ standard Brownian motions, independent of the $X_{i,n}$.
Clearly, since $\Psi_{i,n}$ is \pE{chi-square}   distributed with $2m$ degrees of freedom, it holds for any $q > 1$ that
\begin{align*}
  \lim_{n\to\infty} n\P( \phi_{i,n} > q ) &=\lim_{n\to\infty} n\P\left( \kE{\Psi_{i,n}}  > 2b_nq^2 \right)\\
  &\leq \lim_{n\to\infty} n K \exp(-b_nq^2)
   = 0,
\end{align*}
where $K>0$ is a constant. Furthermore, for arbitrary $\epsilon,r > 0$ we have
\begin{align}
\P\;(& |1 - \phi_{i,n}| > \epsilon | X_{i,n} \in [-r,r])  \label{fraction}\\
  & = \P\left( \kE{\Psi_{i,n}} \notin [2b_n(1-\epsilon)^2,2b_n(1+\epsilon)^2]\, \Big |\, \sum_{j=1}^m B_{i,j}(1)\widetilde{B}_{i,j}(1) \in [b_n -r,b_n +r]\right) \nonumber\\
 &= \frac{\P\left( \kE{\Psi_{i,n}}\notin [2b_n(1-\epsilon)^2,2b_n(1+\epsilon)^2]\, , \, \sum_{j=1}^m B_{i,j}(1)\widetilde{B}_{i,j}(1) \in [b_n -r,b_n +r]\right)}
  {\P\left( \sum_{j=1}^m B_{i,j}(1)\widetilde{B}_{i,j}(1) \in [b_n -r,b_n +r]\right)}.\nonumber
\end{align}
By Lemma \ref{lemProducts}, for large $n\in\N$ the denominator can be bounded by
\begin{align}
  \label{num}\P\left( \sum_{j=1}^m B_{i,j}(1)\widetilde{B}_{i,j}(1) \in [b_n -r,b_n +r]\right) \geq
  K' \left( (b_n - r)^{m/2-1} e^r  - (b_n + r)^{m/2-1} e^{-r}\right) e^{-b_n},
\end{align}
for some constant $K'>0$.
For the numerator we first note that
\begin{align*}
\kE{\Psi_{i,n}}=  \sum_{j=1}^m \left(B^2_{i,j}(1) + \widetilde{B}^2_{i,j}(1)\right) \geq 2 \sum_{j=1}^m B_{i,j}(1)\widetilde{B}_{i,j}(1)
\end{align*}
and thus for $n$ large enough it suffices to consider
\begin{align}
  \notag\P&\left( \kE{\Psi_{i,n}} > 2b_n(1+\epsilon)^2 \, , \, 2\sum_{j=1}^m B_{i,j}(1)\widetilde{B}_{i,j}(1) \in [2b_n -2r,2b_n + 2r]\right)\\
  \notag&\leq \P\left( \sum_{j=1}^m \left(B_{i,j}(1) + \widetilde{B}_{i,j}(1)\right)^2 > 2b_n(1+\epsilon)^2 + 2b_n - 2r \right)\\
  \notag&\leq \P\left( 2 \chi^2_m > 4b_n(1+\epsilon) - 2r \right)\\
  \label{denom}& \leq K'' (2b_n(1+\epsilon) - r)^{m/2-1} e^{-b_n(1+\epsilon)},
\end{align}
where $\chi^2_m$ is a \pE{chi-square} distribution with $m$ degrees of freedom and $K''>0$ is a constant.
From \eqref{num} and \eqref{denom} it is now obvious, that the probability in
\eqref{fraction} turns to $0$, as $n\to\infty$.
Thus, assumption \ref{assump2} of Lemma \ref{main_lemma} is fulfilled.

Note that $\delta_{i,n}$ in \eqref{split2} is independent of $X_{i,n}$. For any $\epsilon > 0$
\begin{align*}
  \P( \| \delta_{1,n}\| > \epsilon ) &=
  \P\left(\sup_{t\in[0,1]}  \left|\sum_{j=1}^m B_{i,j}^*(t)\widetilde B_{i,j}^*(t)\right|
  > 2b_n\epsilon  \right)\\
  &\leq m \P\left(  \sup_{t\in[0,1]}\left|B_{i,1}^*(t)\right| \sup_{t\in[0,1]}\left| \widetilde B_{i,1}^*(t)\right|
  > 2b_n\epsilon/m  \right)\\
  &\leq M \exp\left(- \frac{b_n\epsilon}{m}\right)\\
  &= M' n^{-\epsilon/m} (\ln(n))^{-(m/2 - 1)\epsilon/m)}\to 0,\quad n\to\infty,
\end{align*}
where the second inequality follows from {Lemma \ref{lem0} in Appendix}
and $M,M'$ are positive constants. Clearly, for $C > m$ we further have
\begin{align*}\label{delta_sec}
  n\P( \| \delta_{1,n}\| > C )  \to 0,\quad n\to\infty.
\end{align*}
Thus assumption \ref{assump3} of Lemma \ref{main_lemma} holds, and the proof is complete.
\end{proof}

{\section{Conclusion and further work} \label{sec:conclude}

Brown-Resnick processes have gained a lot of attention recently both because of their  theoretical intricacies as well as their potential applicability, especially in space-time modeling of
extreme events; see \citet{davison:padoan:ribatet:2012}. To this end, it is an important
fact that this class of max-stable processes naturally appears as max-limits of
Gaussian processes (cf.\ \citet{kabluchko:schlather:dehaan:2009},\citet{kabluchko:2011}).
We have shown that these processes appear more generally as limits of maxima of not only Gaussian,
but also squared Bessel processes and Brownian scalar product processes.
Further generalizations are under investigation. A recent work  by \citet{engelkeetal:2012}
shows that for instance that H\"usler-Reiss type \pE{limit distributions} are obtained for non-identically
distributed independent  Gaussian random vectors.
A natural extension could be thus to consider maxima of non-identically distributed  independent
Gaussian processes and their functional limits. Furthermore, the independence assumption
can be \pE{eventually} relaxed regarding the paper as in \citet{Hashorva:Weng}, so that the limit process still remains Brown-Resnick.
In a different direction,  there has been some developments in simulating Brown-Resnick
processes \citep{engelke:kabluchko:Schlather:2012,oestingetal:2012}.
An alternative formulation as the limit of other processes  as described in this paper
can potentially lead to further techniques for simulation.

\section*{Acknowledgement}
The authors are grateful to an anonymous referee for several
comments and corrections that considerably improved the paper.
The authors kindly acknowledge partial support by the Swiss National Science Foundation grants 200021-1401633/1, 200021-134785  and
 the project RARE -318984 (a Marie Curie FP7 IRSES Fellowship).
 Bikramjit Das would also like to thank RiskLab for partial financial support.

\section*{Appendix}\label{sec:appendix}

\BL\label{lem0}  Let $X_{i}, i=1,2$, be two positive independent random variables such that
 \BQN\label{cc}
  \P\left(X_i> x\right) =(1+o(1)) C_i x^{\alpha_i} \exp(- L_i x^{p_i}),
 \EQN
 with $L_i,C_i,p_i,i=1,2$ positive constants, $\alpha_1,\alpha_2\in \R$. Then
 as $x\rightarrow \infty $ we have
\BQN
\label{qTang}  \P\left(X_1X_2> x\right)&=&(1+o(1))
\Bigl(\frac{2\pi p_{2}L_{2}}{p_{1}+p_{2}}\Bigr)%
^{1/2}C_{1}C_{2}A^{p_{2}/2+\alpha _{2}-\alpha _{1}}x^{\frac{2p_{2}\alpha
_{1}+2p_{1}\alpha _{2}+p_{1}p_{2}}{2(p_{1}+p_{2})}} \notag\\
&&\times \exp \left( -(L_{1}A^{-p_{1}}+L_{2}A^{p_{2}})x^{\frac{p_{1}p_{2}}{%
p_{1}+p_{2}}}\right),
\EQN
where $ A =   [(p_1L_1)/(p_2 L_2)]^{1/(p_1+p_2)}$. If further $X_1$ possesses a density  $h_1$ which is bounded and ultimately decreasing
such that $h_1(x)=(1+o(1)) L_1p_1 x^{p_1-1} \P\left(X_1> x\right)$ as $x\to \IF$, then $X_1X_2$ possesses a density $h$ satisfying
 \BQN\label{resDens} h(x) &=& (1+o(1)) L_1p_1 A^{-p_1} x^{p_1p_2/(p_1+p_2)- 1} \P\left(X_1X_2> x\right), \quad x\to \IF.
\EQN
\EL

\begin{proof}
The tail asymptotics of $X_1X_2$ is \dB{proved} in \citet{arendarczyk:debicki:2011} whereas \eqref{resDens} is given in \cite{HAWengStatistics2013}, Corollary 2.2.  An alternative somewhat shorter proof of \eqref{qTang} is derived using the following arguments:
We have  for some $0<l_{1}<1 < l_{2}<\infty $  and  $z_{x} = Ax^{{p_{1}}/{(p_{1}+p_{2})}}, A = [(p_1 L_1)/(p_2 L_2)]^{1/(p_1+p_2)} $
\begin{eqnarray*}
\P({X_1X_2}> x)&{\sim}&C_{1}C_{2}p_{2}L_{2}x^{\alpha
_{1}}z_{x}^{p_{2}+\alpha _{2}-1-\alpha _{1}}\int_{l_{1}Ax^{\frac{p_{1}}{%
p_{1}+p_{2}}}}^{l_{2}Ax^{\frac{p_{1}}{p_{1}+p_{2}}}}\exp (
-L_{1}\left( x/y\right) ^{p_{1}}-L_{2}y^{p_{2}}) \mathrm{d}y \\
&{\sim}& C_{1}C_{2}p_{2}L_{2}x^{\alpha
_{1}}z_{x}^{p_{2}+\alpha _{2}-\alpha _{1}}\int_{l_{1}}^{l_{2}}\exp (
-L_{1}(x/z_{x})^{p_{1}}\left( 1/y\right)
^{p_{1}}-L_{2}z_{x}^{p_{2}}y^{p_{2}}) \mathrm{d}y \\
&{=}& C_{1}C_{2}p_{2}L_{2}x^{\alpha
_{1}}z_{x}^{p_{2}+\alpha _{2}-\alpha _{1}}\int_{l_{1}}^{l_{2}}\exp (
-x^{\frac{p_{1}p_{2}}{p_{1}+p_{2}}%
}[L_{1}A^{-p_{1}}y^{-p_{1}}+L_{2}A^{p_{2}}y^{p_{2}}]) \mathrm{d}y.
\end{eqnarray*}%
Since the function $\psi (y)=L_{1}A^{-p_{1}}y^{-p_{1}}+L_{2}A^{p_{2}}y^{p_{2}}$ {attains} its minimum in $[l_1,l_2]$
at $1$, applying the Laplace approximation we obtain
\begin{eqnarray*}
\int_{l_{1}}^{l_{2}}\exp (-x^{\frac{p_{1}p_{2}}{p_{1}+p_{2}}%
}[L_{1}A^{-p_{1}}y^{-p_{1}}+L_{2}A^{p_{2}}y^{p_{2}}]) \mathrm{d}y {\sim} \frac{\sqrt{2\pi }}{\sqrt{x^{\frac{p_{1}p_{2}}{p_{1}+p_{2}}}\psi ^{\prime
\prime }(1)}}\exp \left( -\psi (1)x^{\frac{p_{1}p_{2}}{p_{1}+p_{2}}%
}\right),\quad x\rightarrow \infty ,
\end{eqnarray*}%
where
\begin{equation*}
\psi (1)=L_{1}[(p_{1}L_{1})/(p_{2}L_{2})]^{-%
\frac{p_{1}}{p_{1}+p_{2}}}+L_{2}[(p_{1}L_{1})/(p_{2}L_{2})]^{\frac{p_{2}}{%
p_{1}+p_{2}}}, \quad
\psi ^{\prime }(1)=0, \quad \psi ^{\prime \prime}(1)=L_{2}A^{p_{2}}p_{2}(p_{1}+p_{2})>0,
\end{equation*}
hence the claim follows.
\end{proof}

\BL\label{lemProducts}
If $X_i,Y_i,i\ge 1$, are independent $N(0,1)$ Gaussian random variables, then for any positive integer $k,m$, as $x\to \IF$
we have
\BQN
   \P\left( \sum_{i=1}^m X_i Y_i > x\right) &=&  \frac{ (1+o(1))}{2^{ m/2} \Gamma(m/2)} x^{m/2-1} \exp(-x)
   = (1+o(1)) f_{m}(x),
\EQN
\BQN\label{gamma}
  \P\left(\sum_{i=1}^m X_i^2  > x\right) &=&  \frac{ (1+o(1))}{2^{ m/2-1} \Gamma(m/2)} x^{m/2-1} \exp(-x/2) = (1+o(1))
  2 g_m(x),
\EQN
where $f_m$ and $g_m$ are the densities of $\sum_{i=1}^m X_i Y_i$ and $\sum_{i=1}^m X_i^2$, respectively.
Furthermore $\sum_{i=1}^m X_i Y_i$ is in the Gumbel max-domain of attraction with norming constants
$$ a_n^*= 1, \quad b_n^* = 
\ln n + (m/2- 1) \ln (\ln n) - (m/2 -1)\ln  2  - \ln \Gamma(m/2).$$

\EL
\begin{proof} The proof follows from Example 5 and 6 in \citet{HashKorshPit}. We give here a direct proof utilizing Lemma \ref{lem0}. Since  $\sum_{i=1}^m X_i Y_i $ is symmetric about 0 and
$$ \abs{\sum_{i=1}^m X_i Y_i } \equaldis \abs{X_1 \sqrt{ \sum_{i=1}^m Y_i^2}}=: \abs{X_1} \abs{ Z},$$
with $Z$ being independent of $X_1$. The asymptotic behavior of $Z^2$ follows immediately from the properties of Gamma random variables as given in
\eqref{gamma}. {Consequently,} applying Lemma \ref{lem0} we obtain 
\BQNY
  \P\left( \sum_{i=1}^m X_i Y_i > x \right) &=& \frac 12
  \P\left( \abs{\sum_{i=1}^m X_i Y_i} > x\right)\\
&=& \frac 1 2    \P\left({ X_1^2 Z^2 > x^2}\right)\\
&=& (1+o(1)) \frac{ 1}{2^{ m/2} \Gamma(m/2)} x^{m/2-1} \exp(-x), \quad x\to \IF,
\EQNY
and thus the norming constants can be easily found; see e.g.,  \cite[p.155]{embrechts:kluppelberg:mikosch:1997}, hence the claim follows.
\end{proof}

\bibliographystyle{model1b-num-names}
\bibliography{bibfilenew}
\end{document}